\documentclass[12pt]{amsart}
\usepackage{amsmath,amssymb,amsbsy,amsfonts,latexsym,amsopn,amstext,cite,
                                               amsxtra,euscript,amscd,bm,mathabx}
\usepackage{url}
\usepackage[colorlinks,linkcolor=blue,anchorcolor=blue,citecolor=blue,backref=page]{hyperref}
\usepackage{color}
\usepackage{graphics,epsfig}
\usepackage{graphicx}
\usepackage{float} 
\usepackage[english]{babel}
\usepackage{mathtools}
\usepackage{todonotes}
\usepackage{url}
\usepackage[colorlinks,linkcolor=blue,anchorcolor=blue,citecolor=blue,backref=page]{hyperref}

\usepackage[norefs,nocites]{refcheck}

\hypersetup{breaklinks=true}

\usepackage[norefs,nocites]{refcheck}
\usepackage[english]{babel}
\begin{document}

\newtheorem{thm}{Theorem}
\newtheorem{lem}[thm]{Lemma}
\newtheorem{claim}[thm]{Claim}
\newtheorem{cor}[thm]{Corollary}
\newtheorem{prop}[thm]{Proposition} 
\newtheorem{definition}[thm]{Definition}
\newtheorem{rem}[thm]{Remark} 
\newtheorem{question}[thm]{Open Question}
\newtheorem{conj}[thm]{Conjecture}
\newtheorem{prob}{Problem}

\newtheorem{lemma}[thm]{Lemma}

\newcommand{\GL}{\operatorname{GL}}
\newcommand{\SL}{\operatorname{SL}}
\newcommand{\lcm}{\operatorname{lcm}}
\newcommand{\ord}{\operatorname{ord}}
\newcommand{\Op}{\operatorname{Op}}
\newcommand{\Tr}{\operatorname{Tr}}
\newcommand{\Nm}{\operatorname{Nm}}

\numberwithin{equation}{section}
\numberwithin{thm}{section}
\numberwithin{table}{section}

\def\vol {{\mathrm{vol\,}}}
\def\squareforqed{\hbox{\rlap{$\sqcap$}$\sqcup$}}
\def\qed{\ifmmode\squareforqed\else{\unskip\nobreak\hfil
\penalty50\hskip1em\null\nobreak\hfil\squareforqed
\parfillskip=0pt\finalhyphendemerits=0\endgraf}\fi}

\def \balpha{\bm{\alpha}}
\def \bbeta{\bm{\beta}}
\def \bgamma{\bm{\gamma}}
\def \blambda{\bm{\lambda}}
\def \bchi{\bm{\chi}}
\def \bphi{\bm{\varphi}}
\def \bpsi{\bm{\psi}}
\def \bomega{\bm{\omega}}
\def \btheta{\bm{\vartheta}}

\newcommand{\bfxi}{{\boldsymbol{\xi}}}
\newcommand{\bfrho}{{\boldsymbol{\rho}}}

\def\Kab{\sfK_\psi(a,b)}
\def\Kuv{\sfK_\psi(u,v)}
\def\SaUV{\cS_\psi(\balpha;\cU,\cV)}
\def\SaAV{\cS_\psi(\balpha;\cA,\cV)}

\def\SUV{\cS_\psi(\cU,\cV)}
\def\SAB{\cS_\psi(\cA,\cB)}

\def\Kmnp{\sfK_p(m,n)}

\def\KKap{\cH_p(a)}
\def\KKaq{\cH_q(a)}
\def\KKmnp{\cH_p(m,n)}
\def\KKmnq{\cH_q(m,n)}

\def\Klmnp{\sfK_p(\ell, m,n)}
\def\Klmnq{\sfK_q(\ell, m,n)}

\def \SALMNq {\cS_q(\balpha;\cL,\cI,\cJ)}
\def \SALMNp {\cS_p(\balpha;\cL,\cI,\cJ)}

\def \SACXMQX {\fS(\balpha,\bzeta, \bxi; M,Q,X)}

\def\SAMJp{\cS_p(\balpha;\cM,\cJ)}
\def\SAMJq{\cS_q(\balpha;\cM,\cJ)}
\def\SAqMJq{\cS_q(\balpha_q;\cM,\cJ)}
\def\SAJq{\cS_q(\balpha;\cJ)}
\def\SAqJq{\cS_q(\balpha_q;\cJ)}
\def\SAIJp{\cS_p(\balpha;\cI,\cJ)}
\def\SAIJq{\cS_q(\balpha;\cI,\cJ)}

\def\RIJp{\cR_p(\cI,\cJ)}
\def\RIJq{\cR_q(\cI,\cJ)}

\def\TWXJp{\cT_p(\bomega;\cX,\cJ)}
\def\TWXJq{\cT_q(\bomega;\cX,\cJ)}
\def\TWpXJp{\cT_p(\bomega_p;\cX,\cJ)}
\def\TWqXJq{\cT_q(\bomega_q;\cX,\cJ)}
\def\TWJq{\cT_q(\bomega;\cJ)}
\def\TWqJq{\cT_q(\bomega_q;\cJ)}

 \def \xbar{\overline x}
  \def \ybar{\overline y}

\def\cA{{\mathcal A}}
\def\cB{{\mathcal B}}
\def\cC{{\mathcal C}}
\def\cD{{\mathcal D}}
\def\cE{{\mathcal E}}
\def\cF{{\mathcal F}}
\def\cG{{\mathcal G}}
\def\cH{{\mathcal H}}
\def\cI{{\mathcal I}}
\def\cJ{{\mathcal J}}
\def\cK{{\mathcal K}}
\def\cL{{\mathcal L}}
\def\cM{{\mathcal M}}
\def\cN{{\mathcal N}}
\def\cO{{\mathcal O}}
\def\cP{{\mathcal P}}
\def\cQ{{\mathcal Q}}
\def\cR{{\mathcal R}}
\def\cS{{\mathcal S}}
\def\cT{{\mathcal T}}
\def\cU{{\mathcal U}}
\def\cV{{\mathcal V}}
\def\cW{{\mathcal W}}
\def\cX{{\mathcal X}}
\def\cY{{\mathcal Y}}
\def\cZ{{\mathcal Z}}
\def\Ker{{\mathrm{Ker}}}

\def\NmQR{N(m;Q,R)}
\def\VmQR{\cV(m;Q,R)}

\def\Xm{\cX_m}

\def \A {{\mathbb A}}
\def \B {{\mathbb A}}
\def \C {{\mathbb C}}
\def \F {{\mathbb F}}
\def \G {{\mathbb G}}
\def \L {{\mathbb L}}
\def \K {{\mathbb K}}
\def \PP {{\mathbb P}}
\def \Q {{\mathbb Q}}
\def \R {{\mathbb R}}
\def \Z {{\mathbb Z}}
\def \fS{\mathfrak S}

\def\e{{\mathbf{\,e}}}
\def\ep{{\mathbf{\,e}}_p}
\def\eq{{\mathbf{\,e}}_q}

\def\\{\cr}
\def\({\left(}
\def\){\right)}
\def\fl#1{\left\lfloor#1\right\rfloor}
\def\rf#1{\left\lceil#1\right\rceil}

\def\Tr{{\mathrm{Tr}}}
\def\Nm{{\mathrm{Nm}}}
\def\Im{{\mathrm{Im}}}

\def \oF {\overline \F}

\newcommand{\pfrac}[2]{{\left(\frac{#1}{#2}\right)}}

\def \Prob{{\mathrm {}}}
\def\e{\mathbf{e}}
\def\ep{{\mathbf{\,e}}_p}
\def\epp{{\mathbf{\,e}}_{p^2}}
\def\em{{\mathbf{\,e}}_m}

\def\Res{\mathrm{Res}}
\def\Orb{\mathrm{Orb}}

\def\vec#1{\mathbf{#1}}
\def \va{\vec{a}}
\def \vb{\vec{b}}
\def \vs{\vec{s}}
\def \vu{\vec{u}}
\def \vv{\vec{v}}
\def \vz{\vec{z}}
\def\flp#1{{\left\langle#1\right\rangle}_p}
\def\T {\mathsf {T}}

\def\sfG {\mathsf {G}}
\def\sfK {\mathsf {K}}

\def\mand{\qquad\mbox{and}\qquad}

%\title[Additive energy of cyclic matrix groups]
%{Additive energy of cyclic matrix groups and character 
% sums with matrix exponential functions}

\title[Matrix powers, Kloosterman 
sums,   quantum ergodicity]
{Equations  and character sums with matrix powers, Kloosterman 
sums over small subgroups and  quantum ergodicity}

\author[A. Ostafe] {Alina Ostafe}
\address{School of Mathematics and Statistics, University of New South Wales, Sydney NSW 2052, Australia}
\email{alina.ostafe@unsw.edu.au}

\author[I. E. Shparlinski] {Igor E. Shparlinski}
\address{School of Mathematics and Statistics, University of New South Wales, Sydney NSW 2052, Australia}
\email{igor.shparlinski@unsw.edu.au}

 \author[J. F. Voloch] {Jos\'e Felipe Voloch}

\address{
School of Mathematics and Statistics,
University of Canterbury,
Private Bag 4800, Christchurch 8140, New Zealand}
\email{felipe.voloch@canterbury.ac.nz}

\begin{abstract}  We obtain a nontrivial bound on the number of solutions to the equation 
$$A^{x_1} + \ldots + A^{x_\nu} = A^{x_{\nu+1}} + \ldots +  A^{x_{2\nu}}, \quad 1 \le x_1, \ldots,x_{2\nu} \le \tau,
$$ with a fixed  $n\times n$ matrix $A$ over  a finite field $\F_q$ of $q$ elements of multiplicative order $\tau$. We give applications of our result to obtaining a new bound of additive character sums with a matrix exponential function, which is nontrivial beyond the square-root threshold. For $n=2$ this equation has been considered by Kurlberg and Rudnick (2001) (for $\nu=2$) and Bourgain (2005) (for large $\nu$) in their study of quantum ergodicity for linear maps over residue rings.  Here we use a new approach to 
improve their results. We also  obtain a bound on Kloosterman sums over small subgroups, of size  below the square-root threshold.
\end{abstract}



\maketitle

\tableofcontents

%---------------------------------------------------------------------
\section{Introduction}
\subsection{Set-up and motivation}
For a positive integer $n$ and a prime power $q$  we use $\GL(n,q)$ and $\SL(n,q)$ to denote the
general and special linear groups 
of $n\times n$ matrices over the finite field $\F_q$ of $q$ elements, respectively.

For a matrix $A \in \GL(n,q)$ and a positive integer $\nu$,  we denote by 
$Q_{\nu,n,q}(A)$ the number of solutions  to the matrix equation 
\begin{equation}
\label{eq:def Q}
A^{x_1} + \ldots +A^{x_\nu} = A^{x_{\nu+1}}  + \ldots + A^{x_{2\nu}} , \quad 1 \le x_1\ldots, x_{2\nu}\le \tau,
\end{equation}
where $\tau$ is the multiplicative order of $A$, that is, the smallest $t\ge1$ such that $A^t$ is the identity  matrix. 
We also set 
\begin{equation}
\label{eq:def E F}
E_{n,q}(A) = Q_{2,n,q}(A) \mand F_{n,q}(A) = Q_{3,n,q}(A). 
\end{equation}
In particular, the quantity
$E_{n,q}(A)$ is called   the {\it additive energy\/} of the multiplicative subgroup $\langle A\rangle$ generated by $A$
in the ring of $n\times n$ matrices, see~\cite{TaVu} for a background and exposition of the
role of additive energy.

 We first recall that for $n=1$,  that is, in the scalar case, a variety of bounds on the additive 
 energy of multiplicative subgroups of  $\F_p^*$ can be found in~\cite{HBK,MRSS,Shkr1,Shkr2}, the case 
 of arbitrary finite fields is   more involved~\cite{Moh,Zhel}. 
 
Furthermore, for $n =2$ , a prime $q=p$ and  a matrix $A \in \SL(2,p)$, 
 Kurlberg and Rudnick~\cite{KR} have shown links between such results and 
 the problem of equidistribution of eigenfunctions  of the ``quantised cat map'', which is a toy model of {\it quantum chaos\/}, we refer to~\cite{GuHa,Kel,KRR,Kurl,KR0,KR,KR2,Ros} for further references and concrete results. Bourgain~\cite{Bourg2}
has given a stronger version of~\cite{KR}. Here we obtain a further improvement 
of~\cite{Bourg2,KR} 
and using a different approach,  give an explicit version of the bound of Bourgain~\cite[Theorem~3]{Bourg2}. 

This motivates to study the case of arbitrary $n$ and also of more general fields 
and matrices. It is also well-known that such  estimates lead to new bounds of 
exponential sums and thus in turn apply to some additive problems. We present such 
applications as well. 

Furthermore, we apply our results to obtain an explicit bound on Kloosterman 
sums over small subgroups of $\F_p^*$ for a prime $p$. While general results of Bourgain~\cite{Bourg1} apply to very small subgroups, they are not explicit and 
making them explicit appears to be very nontrivial. Thus, till now, such explicit 
bounds have been known only in the monomial case.

 \subsection{Previous results}
 For $n =2$ and also a prime $q=p$ and  a matrix $A \in \SL(2,p)$, 
 Kurlberg and Rudnick~\cite{KR} have essentially  shown that   
\begin{equation}
\label{eq:KR bound}
E_{2,q}(A) \le 3\tau^2.
\end{equation}

Here we use some ideas which stem from~\cite{CFKLLS} to obtain a nontrivial bound, 
that is, better than $\tau^3$, for any dimension and arbitrary finite field, and we also relax the condition $A \in \SL(n,p)$.

In a higher dimension, also for  $A \in \SL(n,p)$, Bourgain~\cite{Bourg2} used bounds of exponential sums 
 over small subgroups to obtain an asymptotic formula for $Q_{\nu,n,q}(A)$. 
 
Here we obtain new bounds on this quantity, which are based on new estimates 
of $F_{n,q}(A)$ given by~\eqref{eq:def E F}.  It is easy to see that  
$F_{n,q}(A) \le \tau^2 E_{n,q}(A)$, so we are interested in obtaining stronger bounds.
Although we believe our approach can deliver such better estimates for any $n$, 
here, to exhibit the main ideas, we concentrate on the case of $n=2$, a prime $q=p$, 
and $A \in \SL(2,p)$, which corresponds to the settings of Kurlberg and 
Rudnick~\cite{KR}.  In turn, this result leads to 
a new bound on  certain operators considered by Kurlberg and Rudnick~\cite{KR}
and to an explicit form of a result of Bourgain~\cite[Proposition~1]{Bourg2}, see Theorem~\ref{thm:QuanErg}.   

Our approach uses the results of~\cite{StoVol} which 
allows us to obtain nontrivial bounds on the number of rational points on curves of very high degree over finite fields, 
in the regime where the Weil bound (see, 
for example,~\cite[Section~X.5, Equation~(5.2)]{Lor}) becomes trivial. We believe 
this approach is of independent interest and may have several other applications.

  \subsection{Notation}
    We denote by  $ \oF_q$  the algebraic closure of $\F_q$.
For $\lambda \in  \oF_q^{\,*}$, we denote by $\ord \lambda$ the multiplicative order of $\lambda$.

  We note that hereafter in a matrix-vector multiplication we always assume that the 
dimensions are properly matched, that is, the vectors on the left of a matrix are always rows, 
while the vectors on the right of  a matrix are always columns.  In particular $\vu \vv$ is the scalar product of the vectors $\vu$ and $\vv$. 
  
We recall that  the notations $U = O(V)$, $U \ll V$ and $ V\gg U$  
are equivalent to $|U|\leqslant c V$ for some positive constant $c$, 
which through out this work, all implied constants may depend only on $n$.

\section{Main results}

\subsection{Bounds on the number of solutions to matrix equations}
   
We start with the following general bound on $E_{n,q}(A)$ given by~\eqref{eq:def E F}. 
   
\begin{thm}
\label{thm:Cong Bound}
Let  $A\in \GL(n,q)$ be diagonalisable.  
Then we have
$$
E_{n,q}(A) \ll \tau^3 \min\left\{t \tau^{-1/n^2}, t^{n/(n-1)}\tau^{-1/n(n-1)}\right\}
$$
for any $A$ and 
$$
E_{n,q}(A) \ll t\tau^{3-1/n}
$$
if the characteristic polynomial of $A$ is irreducible over $\F_q$, 
where $\tau$ is the multiplicative order of $A$ and $t$  is the multiplicative order of $\det A$.
\end{thm}

In particular, for  $A\in \SL(n,q)$ the bounds of Theorem~\ref{thm:Cong Bound}
become 
$$
E_{n,q}(A) \ll \tau^{3-1/n(n-1)} \mand E_{n,q}(A) \ll  \tau^{3-1/n}
$$ 
for any $A$ and $A$  with an irreducible over $\F_q$ characteristic polynomial, respectively. 

Next we estimate $F_{2,p}(A)$ given by~\eqref{eq:def E F} in the split case. 

\begin{thm}
\label{thm:3-Cong Bound-Split}
Let  $p$ be prime and let $A\in \SL(2,p)$ be diagonalisable with both eigenvalues
in $\F_p$. 
Then we have
$$
F_{2,p}(A) \ll \tau^{11/3}, 
$$
where $\tau$ is the multiplicative order of $A$.
\end{thm}

In the irreducible case we have a slightly weaker bound.

\begin{thm}
\label{thm:3-Cong Bound-Irred}
Let  $p$ be prime and let $A\in \SL(2,p)$ be diagonalisable with both eigenvalues
in $\F_{p^2}^* \setminus \F_p$. 
Then we have
$$
F_{2,p}(A) \ll  \tau^{19/5} + \tau^5 p^{-1}, 
$$
where $\tau$ is the multiplicative order of $A$.
\end{thm}

We note that both Theorems~\ref{thm:3-Cong Bound-Split} and~\ref{thm:3-Cong Bound-Irred} are nontrivial for any $\tau$. Moreover, we note that similarly to~\eqref{eq:KR bound}, we have 
$$
F_{2,q}(A) \ll \tau^4.
$$
Thus, Theorem~\ref{thm:3-Cong Bound-Split} improves this bound for any $\tau$, while Theorem~\ref{thm:3-Cong Bound-Irred} improves it when $\tau< p^{1-\varepsilon}$
for some fixed $\varepsilon > 0$.

   \subsection{Bounds on exponential sums with matrices}
We now use Theorem~\ref{thm:Cong Bound} to obtain a new bound on exponential sums with a matrix exponential function,
which is non-trivial below the square-root threshold, that is, for  $\tau < q^{n/2}$. 
We recall that for $\tau \ge q^{n/2}$ a nontrivial bound, see~\eqref{eq:Kor Bound} below,  can be achieved via  well-known methods
in the case of arbitrary finite fields, which stem from the work of Postnikov~\cite[Chapter~I, Section~4, Lemma~1]{Post}, 
which in turn is a slight variation of a classical result of Korobov~\cite{Kor}.  
We note that both  Korobov~\cite{Kor} and  Postnikov~\cite{Post} formulate their bounds only 
for prime fields, but the proofs extend to arbitrary finite fields without any changes (at the cost of 
essentially only typographical changes). On the other hand,  in the case of prime fields or 
finite fields of large characteristic, several better 
bounds of this kind are known~\cite{BBG, Bourg1, Bourg3, BourChan,  BGK,  DBGGGSST,   Shkr1.5} but their underlying methods, based on 
additive combinatorics, do not extend to arbitrary finite fields (and sometimes apply only to special cases). 

We fix a nontrivial additive character $\psi$ of $\F_q$ and consider the complete character sum
$$
S_{n,q}(\va, \vb;A)= \sum_{x=1}^\tau \psi\(\va A^x \vb\)
$$
with $\va, \vb \in \F_q^n$,  where $\tau$ is the multiplicative order of  $A\in \GL(n,q)$.

We note that it is easy to see that the sequence $\va A^x \vb$ is a linear recurrence sequence, 
and the other way around, any  linear recurrence sequence over $\F_q$ can be represented 
in this for some vectors $\va, \vb$ and a matrix $A$, see~\cite[Section~1.1.12]{EvdPSW}.

We now define
\begin{equation}
\label{eq:kn}
\kappa_n =  \(4n  \fl{ n -   \frac{1}{2}  \fl{\frac{ n-1}{2}}}\)^{-1}. 
\end{equation}
Thus $\kappa_n  \sim (3n^2)^{-1}$, when $n \to \infty$.

\begin{thm}
\label{thm:Exp Bound}  
Let $\va , \vb \in \F_q^n$ and let $A\in \GL(n,q)$ be diagonalisable. Assume that each of the groups of the following $n$ vectors  
$$
\va A^{i}, \qquad i =0, \ldots, n-1,
$$ 
and 
$$
 A^{i} \vb, \qquad i =0, \ldots, n-1,
$$ 
are linearly independent over $\F_q$.  Then, 
we have 
$$
S_{n,q}(\va, \vb;A) \ll
\begin{cases} t^{1/4} \tau^{1/2 - \kappa_n} q^{n/4},& \text{if } \tau> q^{n/2}, \\
 t^{1/4} \tau^{3/4 - \kappa_n} q^{n/8}, & \text{if } \tau\le  q^{n/2}, 
\end{cases} 
$$
where $\tau$ is the multiplicative order of $A$ and $t$  is the multiplicative order of $\det A$.
\end{thm}

Using the idea of Korobov~\cite{Kor}, one can easily obtain the bound 
\begin{equation}
\label{eq:Kor Bound}
|S_{n,q}(\va, \vb;A)| \le q^{n/2},
\end{equation}
which is nontrivial if $\tau \ge q^{n/2+\varepsilon}$ for some fixed $\varepsilon > 0$. The main interest of Theorem~\ref{thm:Exp Bound} is that it remains nontrivial below the square-root threshold, see, for example,~\eqref{eq:tau large}, which is a notoriously difficult range for problems of this flavour. 

We have a stronger bound for matrices with irreducible over $\F_q$ characteristic polynomial,  which also makes redundant the linear independence conditions for both families
of vectors 
$\va A^{i}$ and $A^{i}\vb$, $i =0, \ldots, n-1$.

\begin{thm}
\label{thm:Exp Bound Irred}
Let $\va, \vb \in \F_q^n$ be non-zero vectors and let $A\in \GL(n,q)$ be such that the characteristic polynomial 
 of $A$ is irreducible over $\F_q$.
 Then we have
 $$
S_{n,q}(\va, \vb;A) \ll
\begin{cases} t^{1/4}\tau^{1/2-1/4n}q^{n/4},& \text{if } \tau> q^{n/2}, \\
 t^{1/4} \tau^{3/4-1/4n} q^{n/8}, & \text{if } \tau\le  q^{n/2}, 
\end{cases} 
$$
where $\tau$ is the multiplicative order of $A$ and $t$  is the multiplicative order of $\det A$.
\end{thm}

\begin{rem}  
One can easily check that our method extends, at the cost of only marginal typographical changes,  to the twisted sums 
$$
\sum_{x=1}^\tau \psi\(\va A^x \vb\) \exp\(2 \pi i \alpha x\), \qquad \va, \vb \in \F_q^n, \ \alpha \in\R.
$$
Thus using the standard completing technique, see~\cite[Section~12.2]{IwKow},   one can 
extend the bounds of Theorems~\ref{thm:Exp Bound} and~\ref{thm:Exp Bound Irred} to incomplete 
sums (with just an additional factor $\log q$). 
\end{rem}

Now we obtain stronger results in the case $n=2$.

First we estimate $S_{2,p}(\va, \vb;A)$ in the split case. 

\begin{thm}
\label{thm:Exp Bound-2-Split}  Let  $p$ be prime.
Let $\va,\vb \in \F_p^2$ and let $A\in \SL(2,p)$ be diagonalisable
with eigenvalues in $\F_p^*$. Assume that  in each pair 
$$
\(\va, \va A\) \mand 
\(\vb, A  \vb\)
$$ 
the vectors
are linearly independent over $\F_p$.  Then,  we have 
 $$
S_{2,p}(\va, \vb;A) \ll \min\{\tau^{23/36} p^{1/6},  \tau^{20/27}   p^{1/9} \},
$$ 
where $\tau$ is the multiplicative order of $A$.
\end{thm} 

In the case when the characteristic polynomial 
 of $A$ is irreducible over $\F_p$  we have a weaker result. 
 In fact we formulate it in the setting of matrices $A\in \SL(2,p)$ 
  but with  eigenvalues avoiding $\F_p$.  
 
\begin{thm}
\label{thm:Exp Bound-2-Irred}  Let  $p$ be prime.
Let $\va,\vb \in \F_{p} ^2$ and let $A\in \SL(2,p)$ be diagonalisable
with eigenvalues in $\F_{p^2}^* \setminus \F_p$. Assume that  in each pair 
$$
\(\va, \va A\) \mand 
\(\vb, A  \vb\)
$$   
the vectors
are linearly independent over $\F_{p^2}$.  Then,  we have 
$$
S_{2,p}(\va, \vb;A) \ll   \min\left\{  \tau^{1/2} p^{1/4},  \tau^{13/20} p^{1/6}  , \tau^{34/45}   p^{1/9}  \right\}, 
$$
where $\tau$ is the multiplicative order of $A$.
\end{thm}

Given a multiplicative subgroup $\cG\subseteq \F_q^*$  and $a,b\in \F_q$ we consider Kloosterman sums over a subgroup, 
$$
\sfK_q(\cG; a, b) = \sum_{u \in \cG} \psi\(au + bu^{-1}\). 
$$

It is easy to see that the Weil bound of exponential sums with rational
functions (see, for example,~\cite{MorMor}) implies that 
$$
\sfK_q(\cG; a, b) \ll q^{1/2}, 
$$
unless $a = b = 0$, 
which becomes trivial for subgroups $G$ of order $\tau< q^{1/2}$.

If $g$ is a generator of $\cG\subseteq \F_p^*$ then, applying Theorems~\ref{thm:Exp Bound-2-Split} 
and~\ref{thm:Exp Bound-2-Irred} to 
$$
\va =\(a, 1\), \qquad \vb =  \(1,  b\)^t, \qquad A = \begin{pmatrix} g & 0\\ 0 &g^{-1}\end{pmatrix}, 
$$
(where $\vu^t$ means the transpose of a vector $\vu$) 
we obtain the following result.  

\begin{cor}
\label{cor:Kloost}  Let  $p$ be prime  and let $\cG$  be a 
multiplicative subgroup of  $\F_p^*$ of order $\tau$. Then 
 $$
\sfK_p(\cG; a, b)  \ll \min\{\tau^{23/36} p^{1/6},  \tau^{20/27}   p^{1/9} \}. 
$$
\end{cor}

We now denote by $\cN_q$ the {\it norm subgroup\/} of $\F_q^*$,  that is, 
the subgroup, formed by elements $z \in \F_q$ of norm $\Nm(z) =1$. 

Let $\F_q$ be of characteristic $p$. We recall that, given an additive character $\psi$ of $\F_q$, there exists an element $\alpha \in \F_q$ such that 
$$
\psi(z) = \ep\(\Tr(\alpha z)\),
$$
where 
$$
\ep(u) = \exp(2 \pi i u/p)
$$ 
and $\Tr(z)$ is the trace from $\F_q$ to $\F_p$.

We also observe that for $\lambda \in \cN_{p^2}$ we have $\lambda^{-1} = \lambda^p$. 
Hence
$$
\Tr\(\alpha \lambda^{-1}\) = \Tr\(\alpha \lambda^p\) = \Tr\(\alpha^p \lambda\) . 
$$
Furthermore, it is easy to see that for any $a \in \F_{p^2}^*$  and 
 $\lambda \in \cN_{p^2}$ there are $\va,\vb \in \F_{p}^2$ and $A\in \SL(2,p)$ 
 such that 
\begin{equation}
\label{eq:aAb=Tr}
 \va A^x \vb = \Tr\(a \lambda^{x}\) , \qquad x =1, 2, \ldots . 
\end{equation} 
 Indeed, if $f(X)=X^2-uX+1\in \F_p[X]$  is the minimal polynomial of $\lambda$, that is, 
 $u = \Tr(\lambda) = \lambda + \lambda^{-1}$, then, for 
 $$
 A = \begin{pmatrix} 0 & -1\\ 1 &u\end{pmatrix}, 
$$ 
we see that $f$ is the characteristic polynomial of $A$. Thus, if we define
$$
 \va = \(\Tr(a), \Tr(a\lambda)\), \qquad \vb =  \(1,  0\)^t, 
$$
 both sequences 
$$
 \va A^x \vb \mand \Tr\(a \lambda^{x}\) , \qquad x =0, 1, \ldots , 
 $$
 satisfy the same binary linear recurrence, and one also verifies that they have the same initial values
 $$
 \va I_2 \vb = \Tr\(a\)    \mand  \va A \vb =\Tr\(a \lambda\),
 $$
 where $I_2= A^0$ is the $2\times 2$ identity matrix, 
 which implies~\eqref{eq:aAb=Tr}. Hence,  we see that  Theorem~\ref{thm:Exp Bound-2-Irred}  allows us to estimate 
the Gauss sums over a subgroup:
$$
\sfG_q(\cG; a) = \sum_{u \in \cG} \psi\(au\). 
$$

\begin{cor}
\label{cor:Gauss}  Let  $p$ be prime, $a \in \F_{p^2}^*$ and let $\cG$ be  a
multiplicative subgroup of  $\cN_{p^2}$ of order $\tau$. Then   
$$
\sfG_{p^2} (\cG; a)  \ll  \min\left\{  \tau^{1/2} p^{1/4},  \tau^{13/20} p^{1/6}  , \tau^{34/45}   p^{1/9}  \right\}.
$$
\end{cor}

\begin{rem}  
\label{rem:range} 
Clearly Theorem~\ref{thm:Exp Bound-2-Split} is nontrivial whenever 
$$
\tau> p^{3/7+\varepsilon}
$$
for any fixed $\varepsilon > 0$, while
 Theorem~\ref{thm:Exp Bound-2-Irred} is nontrivial for 
 $$
\tau> p^{5/11+\varepsilon}, 
$$
and similarly for  Corollaries~\ref{cor:Kloost} and~\ref{cor:Gauss}. 
On the other hand, Bourgain~\cite{Bourg1} 
gives  a nontrivial bound of the form $S_{2,p}(\va, \vb;A) \ll \tau^{1-\delta}$ 
provided $\tau> p^{\varepsilon}$ with some $\delta>0$ depending only 
on $\varepsilon$, however this dependence is not explicit and it is not 
obvious how to get such an explicit result. For applications, see 
Section~\ref{sec:quant erg}, 
it is essential to have nontrivial bounds with a power saving for any 
$\tau> p^{1/2}$.
\end{rem}   

\begin{rem}  
Corollary~\ref{cor:Kloost} 
is the first known bound of this type which is nontrivial below the square-root 
threshold.  We note that for Gauss sums
$$
\sfG_p(\cG; a) = \frac{1}{s} \sum_{x\in \F_p^*} \ep\(ax^s\)
$$
where $s = (p-1)/\tau$, the first bound of this type has been 
obtained in~\cite{Shp2.5} and then improved and extended in various 
directions, see~\cite{Bourg1, Bourg3, BourChan,  BouGlib, BGK,  DBGGGSST,  HBK, Kon,  KonShp, Moh, Shp3, Shkr1.5, Zhel} and references therein. However the explicit 
bound of Corollary~\ref{cor:Gauss} is new.
\end{rem}

\section{Applications} 

 \subsection{Additive properties of matrix orbits} 
 In the special case of  $A\in \SL(n,q)$ with  an irreducible characteristic polynomial,  the bound in Theorem~\ref{thm:Exp Bound Irred} 
is nontrivial for   
\begin{equation}
\label{eq:tau large}
\tau \ge q^{n/2 - n/(2n+2) + \varepsilon}
\end{equation}
for any fixed $\varepsilon > 0$ and applies to all nonzero vectors $\va, \vb \in  \F_q^n$. 
This allows 
us some applications to additive properties of orbits of cyclic matrix groups 
$$
\va \langle A\rangle = \{\va A^x:~ x =1, \ldots, \tau\}. 
$$

Our next result is motivated by results of Schoen and Shkredov~\cite[Corollary~49]{SchShk} and 
Shkredov and  Vyugin~\cite[Corollary~5.6]{ShkVyu} on additive properties of small multiplicative 
subgroups in finite fields.

Similarly to~\cite{SchShk,ShkVyu} studying small subgroups,  we are interested in results for 
matrices of small order. 
Furthermore, fo a set  $\cS \subseteq \F_q^n$ and an integer $k\ge 1$, we denote
$$
k \cS = \{\vs_1 + \ldots + \vs_k:~\vs_1 , \ldots , \vs_k \in \cS\}.
$$

To exhibit the main ideas we only consider that case of 
matrices  $A\in \SL(n,q)$ with  an irreducible characteristic polynomial. 

\begin{thm}
\label{thm:Add Basis} Let  $\varepsilon > 0$ be fixed and  let $q$ be sufficiently large. 
Assume that the characteristic polynomial of   $A\in \SL(n,q)$ is irreducible and 
the multiplicative order $\tau$ of $A$  satisfies~\eqref{eq:tau large} for some fixed $\varepsilon > 0$.
Then, for any nonzero vector $\va \in \F_q^n$ and integer 
$$
k> \max\left\{ \frac{2n}{n+1} \varepsilon^{-1}, 3\right\},
$$
we have 
$$
k  \( \va \langle A\rangle\) =   \F_q^n  . 
$$ 
\end{thm}

For $n=2$ and prime $q = p$, we have a stronger results. We recall that, by Remark~\ref{rem:range},  Theorem~\ref{thm:Exp Bound-2-Irred} is nontrivial for 
 $\tau> p^{5/11+\varepsilon}$. 
 
\begin{thm}
\label{thm:Add Basis 2} Let  $\varepsilon > 0$ be fixed and  let $p$ be a sufficiently large prime. 
Assume that the characteristic polynomial of   $A\in \SL(2,p)$ is irreducible and 
the multiplicative order $\tau$ of $A$  satisfies
$$
\tau  \ge p^{5/11 + \varepsilon} 
$$
for some fixed $\varepsilon > 0$.
Then, for any nonzero vector $\va \in \F_p^2$ and integer 
$$
k> \max\left\{ \frac{45}{11} \varepsilon^{-1} - 3 ,   6\right\},
$$
we have 
$$
k  \( \va \langle A\rangle\) =   \F_p^2  . 
$$
\end{thm} 

 \subsection{Quantum ergodicity of linear maps on a torus} 
 \label{sec:quant erg} 
 The goal of~\cite{KR} has  been to show that for almost all $N$, all the eigenfunctions of the ``quantum cat map" become uniformly distributed in a suitable sense. Bourgain~\cite{Bourg2}
has given a quantitative improvement of~\cite{KR} with some unspecified power saving in the 
bounds on the non-uniformity of distribution.   
 
 We now present our improvement of the results of~\cite{Bourg2,KR} and make the 
 bound of  Bourgain~\cite[Theorem~3]{Bourg2} explicit. 
 
   However, 
 before doing this we have to introduce some notation, we refer to~\cite{Bourg2,KR,KR2} 
 for detailed description. 
 
 First, given an integer $N\ge 1$, we introduce the {\it Hilbert space\/} $L^2(\Z_N)$ of functions $\varphi: \Z_N \rightarrow \C$, 
 acting on  the residue ring  $\Z_N = \Z/N\Z$ modulo $N$, and 
 equipped with the  scalar product
 $$
\langle \varphi, \psi\rangle = \frac{1}{N} \sum_{u \in \Z_N}  \varphi(u) \overline{\psi(u)}, 
\qquad \varphi, \psi \in L^2(\Z_N). 
$$

 Next,  we define the family of operators $T_N(\va)$  on  $L^2(\Z_N)$ with 
 $\va=(a_1,a_2) \in \Z^2$, 
 which act as follows:  For a function $\psi \in L^2(\Z_N)$, we have 
 $$
 \(T_N(\va) \psi\)(u) = \exp\( \pi i a_1a_2/N\) 
  \exp\(2 \pi i a_2u/N\) \psi(u+a_1) . 
 $$
 Given a Fourier expansion 
 $$
 f(\vz) = \sum_{\va\in \Z^2} \widehat f(\va) \exp(2\pi \va \vz)
 $$
  of an infinitely differentiable  function $f\in \cC^\infty \(\T_2\)$
defined on a two-dimensional unit torus $\T_2 = \(\R/\Z\)^2$, we define the {\it quantised
operator\/} 
$$
\Op_N(f)  = \sum_{\va\in \Z^2} \widehat f(\va) T_N(\va).
$$
Furthermore, given a matrix 
$$
A =  \begin{pmatrix} a_{11} & a_{12}\\ a_{21} & a_{22}\end{pmatrix} 
 \in \SL(2,\Z)
$$
with distinct real eigenvalues (that is,  with the trace  satisfying $\left|{\mathrm{tr}}\(A\)\right| > 2$) and such that    
$$
a_{11}   a_{12}\equiv a_{21}  a_{22}\equiv 0 \pmod 2, 
$$
once can associate with $A$ a unitary operator $U_N(A)$ called {\it quantised cat map\/}
which satisfies 
$$
U_N(A)^* \Op_N(f)  U_N(A) =  \Op_N(f), 
$$
where $U_N(A)^*= \overline{U_N(A)}^{\,t}$ is the {\it Hermitian transpose\/}  operator. 

%%We also define the scalar product of two complex-valued 
%%functions $\varphi$ and $\psi$ acting on $\Z_N$  as
%%$$
%%\langle \varphi, \psi\rangle = \frac{1}{N} \sum_{u \in \Z_N}  \varphi(u) \overline{\psi(u)}.
%%$$
Finally,  for a function $f\in \cC^\infty \(\T_2\)$ we set 
$$
\Delta(N,f) = \sup_{\psi \in \Psi(A)} \left | \langle  \Op_N(f)\psi, \psi\rangle  
- \int_{\T_2} f(\vv) d \vv\right|, 
$$
where $\Psi(A)$ is the set of all $L^2$-normalised $ \langle \psi, \psi\rangle =1$  eigenfunctions $\psi$  of $U_N(A)$.   

As usual, we say that some property holds for {\it almost all integers $N \ge 1$\/} if when $X\to \infty$, the number of $N \in [1, X]$ for which it fails is $o(X)$. 

Kurlberg and Rudnick~\cite[Theorem~1]{KR} have  proved the 
uniformity of distribution property of quantised 
operators and established that for all functions $f\in \cC^\infty \(\T_2\)$  
we have $\Delta(N,f) \to 0$ as
$N\to \infty$ over a certain sequence consisting of almost all integers. 
Bourgain~\cite[Theorem~3]{Bourg2} gives a quantitative improvement of this result 
and  shows
that there is an absolute constant $\delta > 0$ 
such that for almost all    integers $N \ge 1$, 
for all functions $f\in \cC^\infty \(\T_2\)$ we have 
\begin{equation}
\label{eq:Bourgain bound}
\Delta(N,f) \le N^{-\delta}. 
\end{equation} 
We now obtain an explicit form of~\eqref{eq:Bourgain bound}.

\begin{thm}
\label{thm:QuanErg} For almost all    integers $N \ge 1$, 
for all functions $f\in \cC^\infty \(\T_2\)$ we have 
$$
\Delta(N,f) \le N^{-1/60+ o(1)}. 
$$
\end{thm}

 \section{Preliminaries}

 \subsection{Multiplicative orders of matrices and eigenvalues}
 The following result is perhaps well-known. We supply a short proof for the sake 
 of completeness.

\begin{lemma}
\label{lem: LCM}
Let $A\in \GL(n,q)$ be diagonalisable, and let  $ \lambda_1, \ldots, \lambda_n \in  \oF_q^{\,*}$  be the eigevalues of 
$A\in \GL(n,q)$ (not necessarily distinct).  Then  
$$
\tau= \lcm[\ord \lambda_1, \ldots, \ord\lambda_n],
$$
where $\tau$ is the multiplicative order of $A$. 
\end{lemma}

\begin{proof} Since $A\in \GL(n,q)$ is diagonalisable, there exist $V\in \GL(n,q)$ such that
$$
A=V D V^{-1}, \quad D=\mathrm{diag}(\lambda_1,\ldots,\lambda_n).
$$
Since $\tau$ is  is the multiplicative order of $A$, we have
$$
D^\tau=V^{-1} A^\tau V=I,
$$
which implies that $\lambda_i^\tau=1$, $i=1,\ldots,n$. 

Moreover, if $D^r=I$ for some $r<\tau$, then from the diagonalisation of $A$ we obtain $A^r=I$, which contradicts the definition of $\tau$. 
\end{proof}

 \subsection{Equations with products of exponential functions}

 Here we estimate the number of solutions to the equation 
\begin{equation}
\label{eq:ProdCong}
\prod_{j=1}^n \( \xi_j- \lambda_j^x\) =  \xi_0, \qquad  x =1, \ldots, \tau,  
\end{equation}
with some $ \xi_0, \lambda_1, \ldots, \lambda_n \in  \oF_q^{\,*}$ and $ \xi_1, \ldots,  \xi_n\in\oF_q$, 
where 
$$
\tau= \lcm[\ord \lambda_1, \ldots, \ord\lambda_n].
$$

Clearly, one can now directly use various bounds on the number of solutions to various equations and congruences
with  linear recurrence sequences,
 see, for example,~\cite[Lemma~6]{BFKS}, 
\cite[Lemma~7]{CFKLLS}, \cite[Lemma~9]{CFS}, 
\cite[Proposition~A.1]{CoZa}, \cite[Lemma 6]{FKS},  \cite[Lemma~2]{Shp1},
\cite[Theorem~1]{Shp2}  (some, but not all, of these works are 
also summarised in~\cite[Section~5.4]{EvdPSW}). 

However, we obtain a  better bound via an application of the method of~\cite{CFKLLS}, 
adjusted to the special shape of the equation~\eqref{eq:ProdCong}. In particular, 
our  next result is a multiplicative analogues of~\cite[Lemma~7]{CFKLLS}.

\begin{lemma}
\label{lem:ProdEq}
For 
$n \ge 1$, given 
$ \xi_1, \ldots,  \xi_n \in  \oF_q$  not all  zeros and  $ \xi_0,  \lambda_1, \ldots, \lambda_n \in  \oF_q^{\,*}$,  
we denote by $N$ the number of solutions  to  the equation~\eqref{eq:ProdCong}. 
Suppose  that 
$$
L =  \ord \lambda_1 \ge \ldots \ge  \ord \lambda_n\ge 1 \mand \xi_1 \ne 0. 
$$
Then
$$
N  \ll \tau L^{ - 1/n},
$$
where  
$$
\tau= \lcm[\ord \lambda_1, \ldots, \ord\lambda_n].
$$
\end{lemma}

\begin{proof}  
Extending $\F_q$, without loss of generality we can assume that 
 $\xi_0, \xi_1, \ldots, \xi_n, \lambda_1, \ldots, \lambda_n \in \F_q$.

Let $\vartheta$ be a primitive root of $\F_q$. 
 Putting $\lambda_i =  \vartheta^{r_i}$, we see
that 
 \begin{equation}
\label{eq:NR}
N = \frac{\tau}{q-1} R, 
\end{equation}
where $R$ is the number of solutions of the equation
$$
\prod_{i=1}^{n} \(\xi_i -  \vartheta^{r_iy}\) = \xi_0, \qquad 0 \le y \le q-2.
$$
Let
 $$
D =\min_{1 \le j \le n} \gcd(r_j , q-1) = \gcd(r_1 , q-1) \quad  \text{and}\quad   M = \fl{(q-1)L^{-1/n}}.
$$
In particular, 
$$
L = (q-1)/D.
$$

As in the proof of~\cite[Lemma~7]{CFKLLS}  (however with $n$ instead of $n-1$) we see that by the pigeonhole
principle there exists $\ell$ with
$1 \le \ell \le L-1$ such that
the remainders  $s_i \equiv r_i  \ell \pmod {q-1}$, taken in the
interval $-(q-1)/2 \le s_i \le q/2$, satisfy the inequality
$$
|s_i| \le M, \qquad i =1, \ldots, n.
$$

Let $d = \gcd(\ell, q-1)$. Clearly for any $y$, $0 \le y \le q-2$,
there is a unique representation of the form
$$
y = dw  + \nu, \qquad 0 \le w \le (q-1)/d -1, \  0 \le \nu \le d-1.
$$
We now define $z$ by the congruence 
$$
(\ell/d) z \equiv w \pmod {(q-1)/d}, \qquad  0 \le z \le (q-1)/d - 1.
$$
Therefore we have a unique representation of the form
$$
y \equiv  \ell z + \nu \pmod{q-1}, \qquad 0 \le z \le (q-1)/d -1, \  0 \le \nu
\le d-1.
$$

Then
 \begin{equation}
\label{eq:RRnu}
R \le \sum_{\nu = 0}^{d-1} R_\nu,
\end{equation}
where
$R_\nu$, $\nu = 0, \ldots\,, d-1$, is the number of solutions of the equation
$$
\prod_{i=1}^{n} \(\xi_i -  \vartheta^{r_i(\ell z + \nu)}  \) = \xi_0,  \qquad 0 \le z \le
(q-1)/d -1.
$$
It is obvious that
 \begin{equation}
\label{eq:RnuQ}
R_\nu = \frac {1}{d} Q_\nu, \qquad \nu = 0, \ldots\,, d-1,
\end{equation}
where $Q_\nu$
is the number of solutions of the exponential  equation
$$
\prod_{i=1}^{n} \(\xi_i -  \vartheta^{r_i \nu}   \vartheta^{s_i z}   \) = \xi_0,   \qquad 0 \le z \le q -2,
$$
which does not exceed the number of zeros of the rational function 
$$
F(U) = \prod_{i=1}^{n} \(\xi_i-  \vartheta^{r_i \nu}  U^{s_i}   \) -\xi_0 \in 
\F_q(U).
$$ 
Since
$L=  \ord  \lambda_1$, 
using the inequality $d D \le (L-1)D < q-1$, one  easily verifies that 
$$
s_1 \equiv   r_1  \ell   \not \equiv 0 \pmod {q-1}
$$
and thus 
$$s_1 \ne 0. 
$$ 
Let  $\zeta$ be any root   of the equation $\xi_1-  \vartheta^{r_1 \nu}  U^{s_1} =0$. 
 Since by our assumption $\xi_1 \ne 0$ we see that $\zeta \ne 0$. 
 Hence, $\zeta$ is not a pole of $F(U)$ and  thus $F(\zeta) =-\xi_0\ne 0$   has  a solution and thus the 
rational function $F(U)$ is not identical to zero. 

 Therefore, $Q_\nu$ does not exceed the number of zeros of a non-zero 
polynomial  of  degree at most
$$
Q_\nu \le n M  =  O\(q L^{-1/n} \) . 
$$ 
Substituting this in~\eqref{eq:RnuQ} and recalling~\eqref{eq:RRnu} we obtain 
$$
R =  O\(q L^{-1/n}\),
$$
which after substitution in~\eqref{eq:NR} implies the result.
\end{proof}

\subsection{Rational points on absolutely irreducible curves}

It is well-known that by the Weil bound we have 
 \begin{equation}
\label{eq:Weil}
\{(x,y)\in \F_q^2:~ F(x,y) = 0\} = q+ O\(d^2 q^{1/2}\)
\end{equation} 
for any absolutely irreducible polynomial  $F(X,Y) \in \F_q[X,Y]$ of degree $d$
(see, for example,~\cite[Section~X.5, Equation~(5.2)]{Lor}).  One can see that~\eqref{eq:Weil} is a genuine asymptotic formula only 
for $d = O(q^{1/4})$ and is in fact weaker that the trivial bound 
$$
\{(x,y)\in \F_q^2:~ F(x,y) = 0\} =  O\(d q\)
$$
for $d \ge q^{1/2}$, which is exactly the range of our interest. 
To obtain nontrivial bounds for such large values of $d$ we use
some idea and results from~\cite{Vol}.   

We start with establishing absolutely irreducibility of polynomials relevant 
to our applications. 

\begin{lemma}
\label{lem:AbsIrred}
For a positive integer $s$ with $\gcd(s,q)=1$ and $a,b \in \F_q$ with $ab(ab-1) \ne 0$, the polynomial
$$
F(X,Y) = ( X^s + Y^s +a) ( X^s + Y^s + bX^s Y^s) -  X^s Y^s \in 
\F_q[X,Y]
$$
is absolutely irreducible.
\end{lemma}

\begin{proof} 
Let us begin by considering the case $s=1$, so 
$$F(X,Y) = ( X + Y +a) ( X + Y + bX Y) -  X Y.$$ 
Since $b \ne 0$, this 
polynomial is a cubic and its homogeneous term of degree $3$ is $bXY(X+Y)$. The curve in the projective plane 
defined by the homogeneisation of $F$ has the points
$$(0:1:0), \ (1:0:0), \ (1:-1:0)$$ 
at infinity. If the curve is not absolutely irreducible,
it has a factor of degree one that passes through one of these points. That means that $F$ has to have a factor of the
form $X-c,Y-c$ or $X+Y-c$ for some constant $c$. We look at each in turn.

First, 
$$F(c,Y) = (bc+1)Y^2 + (2c+bc^2+a+abc-c)Y + c(a+c)$$
and, for this to be identically zero, we must have $c=-1/b \ne 0$
and $c(a+c)=0$, so $c=-a$. Hence $-a=-1/b$, so $ab=1$, contradicting the hypothesis. So $X-c$ is not a factor of $F$ 
and, by symmetry, neither is $Y-c$.

Now,
$$F(X,c-X) =  c(a+c) + (b(a+c)-1)X(c-X)$$
and, for this to be identically zero, we must have $c(a+c) = 0$. If $c=0$, we get $ab = 1$, looking at the coefficient of $X^2$, 
which is a contradiction. 
Otherwise, $c=-a$ and the coefficient of $X^2$ in $F(X,c-X)$ is $1$ so the polynomial is not identically zero
and $X+Y-c$ is not a factor of $F$.

We have shown that, for $s=1$, the polynomial $F$ is absolutely irreducible. We consider the algebraic curve $C$
which is a non-singular projective model of $F=0$ (still with $s=1$). Recall that we assume that $a \ne 0$.

The point $P=(0,-a)$ is a simple point on the curve $F=0$ with 
$$\partial F/\partial X (0,-a) = 0 \mand \partial F/\partial Y (0,-a) = -a \ne 0.$$
So $P$ corresponds to a point on $C$. We consider the functions $x,y$ on $C$ that satisfy the equation $F(x,y)=0$. The function $x$ has a simple zero at $P$, hence is not a power of another function on $C$.
It follows from \cite[Proposition 3.7.3]{Stichtenoth}, that the equation $Z^s = x$ is irreducible over the function field of $C$
and defines a cover $D$ of $C$. Now, consider any point $Q$ on $D$ above the point $(-a,0)$ on $C$. Since $x$ is not zero at $(-a,0)$ (since $a\ne 0$),
the curve $D$ is locally isomorphic to $C$ near $Q$ and we conclude, as above, that the function $y$ on $D$ has a simple zero at $Q$
and, in particular, is not a power of another function on $D$. Again, we conclude that 
the equation $W^s = y$ is irreducible over the function field of $D$ and defines a cover $E$ of $D$. 
In other words, $F(Z^s,W^s)=0$ is an absolutely irreducible equation defining the curve $E$, 
which concludes the proof.  \end{proof}

We first recall the following result~\cite[Theorem~(i)]{Vol}  on the number of points on curves over  $\F_p$. 

 \begin{lem}
\label{lem:HighDeg} Let $p$ be prime and let 
$F(X,Y) \in \F_p[X,Y]$ be an absolutely irreducible polynomial of degree $d$ with 
$d < p$.  
Then
$$
\# \{(x,y) \in \F_p^2~:~F(x,y) = 0\} \le 4d^{4/3} p^{2/3}. 
$$
\end{lem}

We note that in  Lemma~\ref{lem:HighDeg} we dropped  the condition $d > p^{1/4}$ 
of~\cite[Theorem~(i)]{Vol}  since otherwise the Weil bound~\eqref{eq:Weil} 
is stronger.

Unfortunately, Lemma~\ref{lem:HighDeg} applies only to prime fields which is restrictive 
to our applications. However, in the special case of polynomials of our 
interest, we can obtain a version of  Lemma~\ref{lem:HighDeg}, which is suitable 
for such applications.

\begin{lemma}
\label{lem:HighDeg-Fp2} Let $p$ be prime and let  $s= k(p-1)$, where $k$ is a positive 
integer with $\gcd(k,p)=1$.   
Then  for  $a,b \in \F_{p^2} $ with $ab(ab-1) \ne 0$, for  the polynomial
$$
F(X,Y) = ( X^s + Y^s +a) ( X^s + Y^s + bX^s Y^s) -  X^s Y^s \in 
\F_{p^2} [X,Y]
$$
we have
$$
\# \{(x,y) \in \F_{p^2}^2~:~F(x,y) = 0\} \ll s^{6/5}p^{8/5} + p^3. 
$$
\end{lemma}

\begin{proof}  Clearly we can assume that 
\begin{equation}
\label{eq:k small}
k < 3^{-5/4} p
\end{equation}
as otherwise the result is trivial. 

From Lemma~\ref{lem:AbsIrred}, we know that the equation $F(X,Y)=0$ defines an absolutely irreducible curve, of
degree $d=3s$,
denoted by $E$ there. Let $\alpha \in \oF_{p^2}$ satisfy $\alpha^s = -a$. The point $P=(0,\alpha)$
defines a point on $E$ and the line $Y=\alpha$ meets $E$ at $P$ with multiplicity $2s$, since 
$F(X,\alpha) = X^{2s}(1-ab)$. We denote by $x,y$ the functions on $E$ satisfying $F(x,y)=0$.

We want to bound the number $R$ of solutions of $F=0$ in $\F_{p^2}$. We follow the proof of
Lemma~\ref{lem:HighDeg} given  in~\cite[Theorem~(i)]{Vol}. It proceeds by considering, for some
integer $m$, the embedding
of $E$ in $\PP^n$, $n = (m+2)(m+1)/2$, given by the monomials in $X,Y$ of degree at most $m$.
If this embedding is {\it Frobenius classical\/} in the sense of~\cite{StoVol}, then 
by~\cite[Theorem~2.13]{StoVol}
\begin{equation}
\label{eq:stovol}
R \le (n-1)d(d-3)/2 + md(p^2+n)/n.
\end{equation}

If 
\begin{equation}
\label{eq:cond}
m < \min\{ p/2, d-1\} \mand 
p \nmid \prod_{i=1}^m \prod_{j=-m}^{m-i} (2si + j)
\end{equation}
then we claim that
the above embedding is classical. Indeed, the order sequence of the embedding at the point $P$ defined above
consists of the integers $2ni+j, i,j \ge 0, i+j \le m$ as follows by considering the order of vanishing at $P$ of the
functions $x^j(y-\alpha)^i$,  $i,j \ge 0$, $i+j \le m$. The claim now follows 
from~\cite[Corollary~1.7]{StoVol}. 

If the embedding is Frobenius classical, we get the inequality~\eqref{eq:stovol} as mentioned above. If the embedding is
classical but Frobenius nonclassical then, by~\cite[Corollary~2.16]{StoVol}, every rational point of $E$ is a 
Weierstrass point for the embedding and, as the embedding is classical, the number of Weierstrass points satisfies 
\begin{equation}
\label{eq:wp}
R \le n(n+1)d(d-3)/2 + md(n+1). 
\end{equation}

We now choose 
\begin{equation}
\label{eq:m opt}
m = \min\left\{\fl{(p/k)^{1/5}}, 2k-1\right\}.
\end{equation}
If $|i|,|j| \le m$, then for the choice of $m$ as in~\eqref{eq:m opt} we have
\begin{equation}
\label{eq:main ineq}
0 < |-2ki + j| \le 3km \le  3 k^{4/5}p^{1/5} < p
\end{equation}
provided that~\eqref{eq:k small} holds.

Note also that $2si + j \equiv -2ki + j \pmod p$, as $s=k(p-1)$. Hence,
$$
 \prod_{i=1}^m \prod_{j=-m}^{m-i} (2si + j) \equiv 
  \prod_{i=1}^m \prod_{j=-m}^{m-i} (-2ki + j) \pmod p.
$$
Thus from the definuition of $m$ in~\eqref{eq:m opt} and the inequalities~\eqref{eq:main ineq} we see that the conditions~\eqref{eq:cond} are satisfied. 
We note that~\eqref{eq:stovol}  and~\eqref{eq:wp}   can be simplified and 
combined as 
$$
R   \ll \max\{m^2 s^2 + sp^2/m,  m^4 s^2\}
\ll sp^2/m +  m^4 s^2.
$$
Since $m \ll (p/k)^{1/5}\ll (p^2/s)^{1/5}$, we have $m^4 s^2 \ll sp^2/m$ and thus
we obtain
$$
R \le sp^2/m.
$$
Recalling the choice of $m$ in~\eqref{eq:m opt}, we obtain the desired result. 
\end{proof} 

We note that for $k \gg p^{1/6}$ the bound of Lemma~\ref{lem:HighDeg-Fp2}
is $O\(s^{6/5}p^{8/5}\)$.

 \section{Proofs of bounds on the number of solutions to matrix equations} 
 
 \subsection{Proof of Theorem~\ref{thm:Cong Bound}}
 Assume that  $\lambda_1, \ldots, \lambda_n$  are eigenvalues of  $A$   in the algebraic closure $ \oF_q$ of $\F_q$ (which 
  are not necessary distinct).   
Multiplying the equation
$$
A^{x_1} + A^{x_2} = A^{x_3} + A^{x_4} , \qquad 1 \le x_1,x_2,x_3,x_4 \le \tau,
$$
by an eigenvector $\vec{v}_i$ for each eigenvalue $\lambda_i$, $i=1,\ldots,n$, we see that  $E_{n,q}(A)$ is the number of solutions to the system of equations
 \begin{equation}
 \begin{split}
 \label{eq: lambda eq}
 \lambda_i^{x_1} +  \lambda_i^{x_2} & =  \lambda_i^{x_3} +  \lambda_i^{x_4}  , \qquad i=1, \ldots, n,\\
 1  & \le x_1,x_2,x_3,x_4 \le \tau. 
\end{split}
\end{equation}

Now, suppose that 
$$
L =  \ord \lambda_1 \ge \ldots \ge  \ord \lambda_n\ge 1.
$$

Hence, the contribution $T_0$ to the number of solutions to~\eqref{eq: lambda eq} from quadruples $(x_1,x_2,x_3,x_4)$ with 
$$ \lambda_1^{x_1} +  \lambda_1^{x_2} = \lambda_1^{x_3} +  \lambda_1^{x_4}=0
$$ 
satisfies 
 \begin{equation}
\label{eq:T0} 
T_0 \ll  \tau^3/L.
\end{equation} 
Indeed, fixing $x_3$ we see that  $x_4$ is uniquely defined 
modulo $\ord \lambda_1 = L$. Now, when $x_3$ and $x_4$ are fixed (in one of  $O(\tau^2/L)$ possible ways), we see from 
Lemma~\ref{lem: LCM} that for each $x_2$ there is a unique value of $x_1$ satisfying~\eqref{eq: lambda eq} and thus we 
obtain~\eqref{eq:T0}.

Therefore, fixing $x_3$ and $x_4$, we now see that it is 
enough to estimate the number of solutions to  at most $\tau^2$ systems of equations 
of the form
 \begin{equation}
\label{eq:syst}
 \lambda_j^{x_1} +  \lambda_j^{x_2}   =\xi_j  , \quad j =1, \ldots, n,\  \qquad 
 1 \le x_1,x_2  \le \tau, 
\end{equation}
with some   $\xi_1, \ldots, \xi_n \in \oF_q$, where $\xi_1 \ne 0$. 

We now derive from~\eqref{eq:syst} that 
\begin{equation}
\label{eq:prod det}
\prod_{j=1}^n (\xi_j -\lambda_j^{x_1}) = \prod_{j=1}^n \lambda_j^{x_2}  = (\det A)^{x_2}.
\end{equation}

If the multiplicative order of $\det A$ is $t$, the equation~\eqref{eq:prod det} is reduced to $t$ equations 
of the form  
$$
\prod_{j=1}^n (\xi_j -\lambda_j^{x_1}) = \prod_{j=1}^n \lambda_j^{x_2}  = \xi_0
$$
for $t$ distinct values $\xi_0 \in \F_q^*$. 
Recalling the bound~\eqref{eq:T0} and Lemma~\ref{lem:ProdEq}, we obtain 
\begin{equation}
\label{eq: T L}
E_{n,q}(A) \ll T_0 + t \tau^3L^{-1/n}  \ll   t \tau^3L^{-1/n}. 
\end{equation}

It remains to obtain a lower bound on $L$. 

First, we observe that 
 \begin{equation}
\label{eq: lcm Ln}
 \lcm[\ord \lambda_1, \ldots, \ord\lambda_n] \le \prod_{i=1}^n \ord \lambda_i \le L^n,
\end{equation}
 and thus by Lemma~\ref{lem: LCM} we have 
 \begin{equation}
\label{eq: L bound 1}
L \ge \tau^{1/n}.
\end{equation}

Second, we have 
$$
\( \prod_{i=1}^n  \lambda_i \)^t = \det A^t  = 1.
$$
Denoting 
$$
\tau_0 =  \lcm[\ord \lambda_1, \ldots, \ord\lambda_{n-1} ] , 
$$
we see that 
$$
\lambda_n^{t \tau_0} =  \( \prod_{i=1}^{n-1} \lambda_i^{\tau_0} \)^{-t} = 1. 
$$
Hence
$$
\ord \lambda_{n} \mid t \tau_0,
$$ 
which implies $\tau \mid  t \tau_0$. Since obviously  $\tau_0  \le L^{n-1}$, we now obtain 
 \begin{equation}
\label{eq: L bound 2}
L \ge (\tau/t)^{1/(n-1)}.
\end{equation}

Substituting the bound~\eqref{eq: L bound 1} and~\eqref{eq: L bound 2} in~\eqref{eq: T L}, we conclude 
the first bound of Theorem~\ref{thm:Cong Bound}.

If the characteristic polynomial of $A$ is irreducible over $\F_q$, then all eigenvalues $\lambda_i$, $i=1,\ldots,n$, have same multiplicative order, which is equal to $\tau$, the order of the matrix $A$. Thus, in the  above we have $L=\tau$ and substituting this  in~\eqref{eq: T L} we conclude the proof. 

 \subsection{Proof of Theorem~\ref{thm:3-Cong Bound-Split}} 
 Let $\lambda, \lambda^{-1} \in \F_p^*$ be the eigenvalues of $A$. In particular, 
 $\tau$ is the multiplicative order of $\lambda$ in $\F_p^*$. 
 
Arguing as in the proof of Theorem~\ref{thm:Cong Bound} we see that 
$F_{2,p}(A)$ is the number of solutions to the system of equations
 \begin{equation}
 \begin{split}
 \label{eq: 3-lambda eq}
 \lambda^{x_1} +   \lambda^{x_2} +  \lambda^{x_3}& =   \lambda^{x_4}+ 
  \lambda^{x_5}+ \lambda^{x_6} \\
   \lambda^{-x_1} +   \lambda^{-x_2} +  \lambda^{-x_3}& =   \lambda^{-x_4}+ 
  \lambda^{-x_5}+ \lambda^{-x_6} \\
 1   \le x_1&, \ldots,  x_6 \le \tau. 
\end{split}
\end{equation}

We fix $x_3, x_4, x_5$ and denote 
$$
a = \lambda^{x_3}- \lambda^{x_4}- \lambda^{x_5}
\mand 
b = \lambda^{-x_3}- \lambda^{-x_4}- \lambda^{-x_5}.
$$
In this notation we can rewrite~\eqref{eq: 3-lambda eq} as
 \begin{equation}
 \begin{split}
 \label{eq: 3-lambda ab}
 \lambda^{x_1} +   \lambda^{x_2} +  a& =     \lambda^{x_6} \\
   \lambda^{-x_1} +   \lambda^{-x_2} + b& =    \lambda^{-x_6} \\
 1   \le x_1, x_2,   x_6 &\le \tau. 
\end{split}
\end{equation}
Multiplying the equations in~\eqref{eq: 3-lambda ab} we obtain 
\begin{equation}
\label{eq: x1x2ab}
\(\lambda^{x_1} +   \lambda^{x_2} +  a\)\(\lambda^{-x_1} +   \lambda^{-x_2} + b\) = 1, 
\qquad 1 \le x_1, x_2 \le \tau. 
\end{equation}

We first consider the case when $ab(ab-1) = 0$. Obviously there are at most $\tau^2$
such choices of  $(x_3, x_4, x_5)$, for each of them the equation~\eqref{eq: x1x2ab} has at most $\tau$ solutions in $(x_1,x_2)$ after which $x_6$ is uniquely defined. 
Hence the total contribution from such solutions is $O(\tau^3)$ which is admissible.

Therefore, from now on we investigate the number $T_{a,b}$ of solutions 
to~\eqref{eq: x1x2ab} for at most $\tau^3$ choices of  $(x_3, x_4, x_5)$,  with 
$$
ab(ab-1) \ne 0.
$$
Let $s = (p-1)/\tau$. 
Clearly $T_{a,b} = s^{-2}R_{a,b}$, where $R_{a,b}$ is the number of solutions to the  equation
$$
\(x^s +  y^s+  a\)\(x^{-s} +   y^{-s} + b\) = 1, \qquad x,y\in \F_p^*,
$$
or, equivalently, to the polynomial equation 
\begin{equation}
\label{eq: Main Eq}
\(x^s +  y^s+  a\)\(x^{s} +   y^{s} + b x^{s} y^{s}\) -x^{s} y^{s} = 0, \qquad x,y\in \F_p^*.
\end{equation}
Applying Lemmas~\ref{lem:AbsIrred} and~\ref{lem:HighDeg},   we obtain 
\begin{equation}
\label{eq: Fin Bound}
T_{a,b} = s^{-2}R_{a,b}\ll s^{-2/3} p^{2/3} \ll \tau^{2/3}.
\end{equation}
After this $x_6$ is uniquely defined. 
Hence the total contribution from such solutions is $O(\tau^{3+2/3})$ which
concludes the proof.

\subsection{Proof of Theorem~\ref{thm:3-Cong Bound-Irred}}
We proceed as in the proof of  Theorem~\ref{thm:3-Cong Bound-Split}  
however with the eigenvalues   $\lambda, \lambda^{-1} \in
\F_{p^2}^* \setminus \F_p$. In particular, we arrive to the same 
equation~\eqref{eq: Main Eq}, however with the variables $x,y \in \F_{p^2}$ 
and with $s = (p^2-1)/\tau$. 
Clearly we can assume that $\lambda$ and  $\lambda^{-1}$ are distinct 
as otherwise $\tau \le 2$ and the bound is trivial. In this case 
$\lambda^{-1} = \lambda^p$ or $\lambda^{p+1} =1$. 
Hence $\tau\mid p+1$ and we see that the conditions of Lemma~\ref{lem:HighDeg-Fp2} 
are satisfied with $s = k(p-1)$ and $k = (p+1)/\tau$. Since for similarly defined quantities 
$T_{a,b}$ and $R_{a,b}$, applying Lemma~\ref{lem:HighDeg-Fp2}, instead of~\eqref{eq: Fin Bound} we obtain 
$$
T_{a,b} = s^{-2}R_{a,b}\ll s^{-4/5} p^{8/5} + p^3 s^{-2} \ll \tau^{4/5} + \tau^2 p^{-1}, 
$$
and the result follows.

 \section{Proofs of bounds on exponential sums with matrices} 

 \subsection{Proof of Theorem~\ref{thm:Exp Bound}}
For any integer $k\ge 1$ we have 
\begin{align*}
S_{n,q}(\va, \vb;A)^k& = \sum_{x_1, \ldots , x_k=1}^\tau 
\psi\(\va \(A^{x_1} + \ldots + A^{x_k}\) \vb\)\\
& = \sum_{\vu \in \F_q^n} \nu_k(\vu) \psi\( \vu \vb\), 
\end{align*} 
where $\nu_k(\vu) $ is the number of solutions to the equation
$$
\va \(A^{x_1} + \ldots + A^{x_k}\)  = \vu, \qquad 1\le x_1, \ldots , x_k\le \tau .
$$

Clearly  $\nu_k(\vu)= \nu_k(\vu A^u) $  for any integer $u$.  Aslo $\vu  A^u$ runs through 
all vectors of $ \F_q^n$ when $\vu$ does. Hence 
\begin{align*}
S_{n,q}(\va, \vb;A)^k &  = \frac{1}{\tau} 
  \sum_{u=1}^\tau \sum_{\vu \in \F_q^n}\nu_k(\vu A^u) \psi\( \vu A^u \vb\)\\
&  = \frac{1}{\tau}    
\sum_{\vu \in \F_q^n}\nu_k(\vu )  \sum_{u=1}^\tau  \psi\( \vu A^u \vb\)\\
& = \frac{1}{\tau}    
\sum_{\vu \in \F_q^n}\nu_k(\vu )  S_{n,q}(\vu,  \vb;A). 
\end{align*} 
Writing 
$$
\nu_k(\vu)  = \nu_k(\vu) ^{1-1/\ell} \(\nu_k(\vu)^2\)^{1/2\ell},
$$ 
by the H{\"o}lder inequality, for any integer $\ell$ we have 
 \begin{equation}
 \begin{split}
 \label{eq: S2kl}
|S_{n,q}(\va, \vb;A)|^{2k\ell}  \le \frac{1}{\tau^{2\ell}} 
\( \sum_{\vu \in \F_q^n}\nu_k(\vu)  \)^{2\ell -2}   
& \sum_{\vu \in \F_q^n}\nu_k(\vu)^2 \\
 &\sum_{\vu \in \F_q^n}   \left| S_{n,q}(\vu,  \vb;A)\right|^{2\ell} .
\end{split}
\end{equation}
Obviously 
 \begin{equation}
 \label{eq:nu1}
 \sum_{\vu \in \F_q^n}\nu_k(\vu)  = \tau^k
\end{equation}
 and 
 \begin{equation}
 \label{eq:nu2}
  \sum_{\vu \in \F_q^n}\nu_k(\vu)^2 = J_k,
\end{equation}
 where $J_k$ is the number of solutions to the equation
\begin{equation}
\begin{split} 
\label{eq:aAxi}
\va \(A^{x_1} + \ldots + A^{x_k} - A^{x_{k+1}} - \ldots - A^{x_{2k}}\)  & = \mathbf{0}, \\ 1\le x_1, \ldots , x_{2k}\le \tau .\qquad \qquad & 
\end{split} 
\end{equation}
Furthermore, by the orthogonality of exponential functions we also have  
 \begin{equation}
 \label{eq:S2l}
 \sum_{\vu \in \F_q^n}   \left| S_{n,q}(\vu,  \vb;A)\right|^{2\ell} 
 = q^n K_\ell , 
\end{equation}
 where $K_\ell$ is the number of solutions to the equation
$$
 \(A^{x_1} + \ldots + A^{x_\ell} - A^{x_{\ell+1}} - \ldots - A^{x_{2\ell}}\)\vb  = \mathbf{0}, \quad 1\le x_1, \ldots , x_{2\ell}\le \tau . 
$$
Substituting~\eqref{eq:nu1}, \eqref{eq:nu2} and~\eqref{eq:S2l} in~\eqref{eq: S2kl}, we derive
  \begin{equation}
 \begin{split}
 \label{eq: S2kl-bound}
|S_{n,q}(\va, \vb;A)|^{2k\ell}  \le q^n \tau^{2k\ell - 2k -2 \ell} J_kK_\ell. 
\end{split}
\end{equation}

Now, if we denote $B=A^{x_1} + \ldots + A^{x_k} - A^{x_{k+1}} - \ldots - A^{x_{2k}}$, for a solution $(x_1,\ldots,x_{2k})$ to~\eqref{eq:aAxi}, we have $\vec{a}B=\mathbf{0}$. Multiplying by powers of $A$, we also obtain the equations
$$
(\vec{a}A^i)B=\mathbf{0}, \qquad i=0,\ldots,n-1,
$$
and thus, 
$$\begin{pmatrix}\vec{a}\\ \vec{a}A \\ \vdots \\ \vec{a}A^{n-1}\end{pmatrix} B=O_n, 
$$
where $O_n$ is the  $n\times n$ zero matrix. 
Since the vectors  
$\va A^{j}$, $j =0, \ldots, n-1$, 
are linearly independent, we obtain that $B=0$. 
Therefore,  
 $J_k$ is the number of solution to the equation
$$
 A^{x_1} + \ldots + A^{x_k} - A^{x_{k+1}} - \ldots - A^{x_{2k}}  = O_n, \quad 1\le x_1, \ldots , x_{2k}\le \tau, 
$$
and similarly for $K_\ell$. 

We now observe that 
 \begin{equation}
\label{eq: I2=J2 = E}
J_2 = K_2 = E_{n,q}(A),
\end{equation}
while we trivially have 
 \begin{equation}
\label{eq: I1=J1 = tau}
J_1 = K_1 = \tau.
\end{equation}

Using~\eqref{eq: S2kl-bound} with $k=2$ and $\ell=1$ we see from~\eqref{eq: I2=J2 = E}
and~\eqref{eq: I1=J1 = tau} that  
  \begin{equation}
 \label{eq: Bound 1}
|S_{n,q}(\va, \vb;A)|  \le q^{n/4} \tau^{ -1/4}E_{n,q}(A)^{1/4}. 
\end{equation}
Similarly, taking $k=\ell=2$ we see from~\eqref{eq: I2=J2 = E} that~\eqref{eq: S2kl-bound}
implies  
 \begin{equation}
 \label{eq: Bound 2}
|S_{n,q}(\va, \vb;A)|  \le q^{n/8} E_{n,q}(A)^{1/4}. 
\end{equation}

At this stage we can already apply Theorem~\ref{thm:Cong Bound} to obtain a nontrivial estimate
on $S_{n,q}(\va, \vb;A)$. However this gives our bound with $\kappa_n$ about $(2n)^{-2}$, 
which we now improve using the argument below. Thus, as we have mentioned,  this argument 
gives $\kappa_n \sim (3n^2)^{-1}$, when $n \to \infty$. 

Let $P_A \in \F_q[X]$  be the characteristic polynomial of $A$. We first note that the bound of Theorem~\ref{thm:Exp Bound} is trivial if $ \tau^{1/4 + \kappa_n} \le q^{n/8}$. Hence we now assume that 
 \begin{equation}
\label{eq: Large tau}
\tau > q^{n/\(2\(1+4 \kappa_n\)\)}. 
\end{equation}

There are some integers $1 \le d_1 < \ldots < d_r$, such that 
we can factor $P_A(X) = g_1(X) \ldots g_r(X) $, where $g_i$ is a product 
of irreducible  over $\F_q$ polynomials of the same degree $d_i$, $i=1, \ldots, r$.
Let $\deg g_i = m_i d_i$, $i=1, \ldots, r$. Thus
 \begin{equation}
\label{eq: sum midi}
 \sum_{i=1}^r m_i d_i = n
\end{equation}
and we also see that  $P_A$ has 
 \begin{equation}
\label{eq: sum mi}
m = \sum_{i=1}^r m_i
\end{equation}
irreducible factors. 
For each of this polynomials we fix one of its roots and denote them 
$\mu_1, \ldots, \mu_m$.

Since  all roots of 
an irreducible over $\F_q$ polynomial have the same multiplicative order, instead of~\eqref{eq: lcm Ln} in the proof of Theorem~\ref{thm:Cong Bound}
we obtain
$$ \lcm[\ord \lambda_1, \ldots, \ord\lambda_n] \le \prod_{j=1}^m \ord \mu_j \le L^m,
$$
 and thus~\eqref{eq: L bound 1} becomes 
 \begin{equation}
\label{eq: L bound 3}
L \ge \tau^{1/m}.
\end{equation}
Hence,  it remains to estimate $m$. 

First we remark that 
$$
\tau \le \lcm\left[q^{d_1}-1, \ldots, q^{d_r}-1\right ]\le q^{d_1+ \ldots + d_r}, 
$$
and together with~\eqref{eq: Large tau} we conclude
 \begin{equation}
\label{eq: sum di}
d_1+ \ldots + d_r \ge \fl{\frac{ n}{2\(1+4 \kappa_n\)}} +1.
\end{equation}

We consider two cases.

First, assume that $d_1=1$.  Then, recalling~\eqref{eq: sum mi} and  
using~\eqref{eq: sum midi}, we obtain
 \begin{equation}
\label{eq: m m1}
m = \sum_{i=1}^r m_i \le m_1 + \frac{1}{2} \sum_{i=2}^r m_id_i  = n/2 +m_1/2.
\end{equation}
On the other hand, from~\eqref{eq: sum di}  we derive 
\begin{align*}
m_1 & = n -  \sum_{i=2}^r m_i d_i \le n - \sum_{i=2}^r d_i  \\
& = n+1 -  \sum_{i=1}^r  d_i \le n -  \fl{\frac{ n}{2\(1+ 4\kappa_n\)}} , 
\end{align*} 
which together with~\eqref{eq: m m1}  implies 
 \begin{equation}
\label{eq: m}
m  \le  \fl{n  - \frac{1}{2}  \fl{\frac{ n}{2\(1+4 \kappa_n\)}} }. 
\end{equation}

Second, if   $d_1 \ge 2$, we obviously have a slightly better bound
$$
m = \sum_{i=1}^r m_i \le   \frac{1}{2} \sum_{i=1}^r m_id_i  = n/2 .
$$

From the definition of $\kappa_n$ given by~\eqref{eq:kn},  we see that  $\kappa_n \le (4n)^{-1}$  and thus
$$
\frac{n}{2} > \frac{n}{2(1 + 4 \kappa_n)}  = \frac{n}{2} - \frac{2n  \kappa_n}{ 1 + 4 \kappa_n} 
> \frac{n-1}{2} . 
$$
Therefore
$$
 \fl{\frac{n}{2(1 + 4 \kappa_n)}} =  \fl{\frac{n-1}{2}}.
$$
Now, using~\eqref{eq:kn} again,  we derive the identity
$$
 \fl{n  - \frac{1}{2}  \fl{\frac{ n}{2\(1+4 \kappa_n\)}} }
 =  \fl{n  - \frac{1}{2}    \fl{\frac{n-1}{2}} } =  \frac{1}{4 n \kappa_n}.
$$
Hence,  the bound~\eqref{eq: m} implies
$$
m \le   \frac{1}{4 n \kappa_n}.
$$ 
Therefore we see from~\eqref{eq: L bound 3} that 
$$
L \ge  \tau^{4 n \kappa_n}, 
$$
which after  substituting in~\eqref{eq: T L} in the proof of Theorem~\ref{thm:Cong Bound}
implies 
$$
E_{n,q}(A) \ll      t \tau^{3-4  \kappa_n}. 
$$ 
and together with~\eqref{eq: Bound 1} and~\eqref{eq: Bound 2} concludes the proof.

 \subsection{Proof of Theorem~\ref{thm:Exp Bound Irred}}
It is enough to show that if the characteristic polynomial $P_A$ of $A$ is irreducible over $\F_q$ then the conditions 
of Theorem~\ref{thm:Exp Bound} are satisfied for any  nonzero vectors $\va, \vb \in \F_q^n$. In this case, exactly as in the proof of Theorem~\ref{thm:Exp Bound} one concludes
from~\eqref{eq: Bound 1} and~\eqref{eq: Bound 2}  and  Theorem~\ref{thm:Cong Bound} 
that the desired bounds hold. 

To show that the conditions of Theorem~\ref{thm:Exp Bound}  on $\va, \vb \in \F_q^n$ 
are now redundant we note that  if  $P_A$ is irreducible over $\F_q$ then all its roots $\lambda_1,\ldots,\lambda_n\in\oF_q$ are distinct, and thus the matrix $A$ is diagonalisable. 

If the vectors $\va A^{i}$, $ i =0, \ldots, n-1$, are linearly dependent over $\F_q$ for some nonzero $\va \in\F_q^n$, then there 
is a linear relation
\begin{equation}
\label{eq:lin dep A}
\sum_{i=0}^{n-1} c_i \va A^{i}=\mathbf{0} 
\end{equation}
for some $c_i\in\F_q$ not all zero. Let $\vec{v}_\ell$ be an eigenvector corresponding to the eigenvalue $\lambda_\ell$, $\ell=1,\ldots,n$, and choose $\ell$ such that $\va$ is not 
orthogonal to  $\vec{v}_\ell$. This is indeed possible since all the vectors $\vec{v}_1,\ldots,\vec{v}_n$ are linearly independent (because they correspond to distinct eigenvalues) 
and thus their span is the whole space $\oF_q^{\,n}$.

Multiplying now~\eqref{eq:lin dep A} on the right with $\vec{v}_\ell$ (which is a column vector), we obtain
$$
\sum_{i=0}^{n-1} c_i \va A^{i}\vec{v}_\ell=\sum_{i=0}^{n-1} c_i \va \lambda_\ell^{i}\vec{v}_\ell=\sum_{i=0}^{n-1} d_{\ell,i} \lambda_\ell^{i}=0,
$$
where $d_{\ell,i}=c_i\vec{a}\vec{v_\ell}\in\oF_q$. Since $\va \vec{v}_\ell\ne 0$ by our choice of $\ell$, and  since not all coefficients $c_i$ are zero, 
we note that $d_{\ell,i}$, $i=0,\ldots,n-1$, are not all zero. Thus, the polynomial 
$$
f(X)=\sum_{i=0}^{n-1}d_{\ell,i}X^i
$$ 
is nonzero. Therefore $\lambda_\ell$ is a roof of $f$ which is of degree $n-1$, and thus impossible  by the irreducibility of  $P_A$. 

Similar argument also applies to the vectors $ A^{i}\vb$, $ i =0, \ldots, n-1$, 
which concludes the proof. 

\subsection{Proof of Theorem~\ref{thm:Exp Bound-2-Split}} We proceed as in the proof of Theorem~\ref{thm:Exp Bound}.  First, as we have noted in~\eqref{eq:KR bound},  for $n=2$ and $t=1$, the argument of Kurlberg and Rudnick~\cite{KR}  
gives 
$$
J_2 = K_2 \ll \tau^2. 
$$ 
Furthermore,  we also have, 
$$
J_3 = K_3 = F_{2,p}(A).
$$
Invoking  Theorem~\ref{thm:3-Cong Bound-Split} and using~\eqref{eq: S2kl-bound} 
with $(k,\ell) =(2,3)$, 
we obtain 
$$
|S_{2,p}(\va, \vb;A)|^{12}  \ll p^2 \tau^{12- 4 -6+ 2 +11/3} = p^2 \tau^{23/3} , 
$$ 
while  the choice  $(k,\ell) =(3,3)$,  gives
$$
|S_{2,p}(\va, \vb;A)|^{18}  \ll p^2 \tau^{18- 6 -6+22/3} = p^2 \tau^{40/3} , 
$$ 
and  the desired bound follows. 

\subsection{Proof of Theorem~\ref{thm:Exp Bound-2-Irred} }
First, we proceed as in the proof of Theorem~\ref{thm:Exp Bound}, thus 
with $(k,\ell) = (2,2)$, using~\eqref{eq:KR bound}, we obtain 
the bound
 $$
S_{2,q}(\va, \vb;A) \ll \tau^{1/2}q^{1/4}. 
$$ 
Note that these bounds hold for any $q$ and also for any characteristic 
polynomial of $A$. It is also easy to check that for $\tau > p^{5/6}$ 
we have 
$$
  \tau^{1/2} p^{1/4} < \min\left\{\tau^{13/20} p^{1/6}, \tau^{34/45}   p^{1/9} \right\}. 
$$
Hence we can now assume that 
 \begin{equation}
\label{eq: large tau}
\tau \le p^{5/6}. 
\end{equation}

Next we proceed as in the proof of Theorem~\ref{thm:Exp Bound-2-Split}
however in an appropriate place we use Theorem~\ref{thm:3-Cong Bound-Irred}
instead of Theorem~\ref{thm:3-Cong Bound-Split}. We also remark that 
$\tau^{19/5} > \tau^5 p^{-1}$ under the assumption~\eqref{eq: large tau}.

Thus, with the choice 
$(k,\ell) =(2,3)$, we obtain 
$$
|S_{2,p}(\va, \vb;A)|^{12}  \ll p^2 \tau^{12- 4 -6+ 2+19/5}   = p^2 \tau^{39/5} , 
$$ 
while  the choice $(k,\ell) =(3,3)$,  gives
$$
|S_{2,p}(\va, \vb;A)|^{18}   \ll p^2 \tau^{18- 6 -6+38/5}   = p^2 \tau^{68/5},  
$$
and  the desired bound follows. 

 \section{Proofs of additive properties of matrix orbits} 

\subsection{Proof of Theorem~\ref{thm:Add Basis}}
 
We have to show that for $k$, which satisfies the inequality  of Theorem~\ref{thm:Add Basis}
and sufficiently large $q$, 
for any $\vu \in \F_q^n$,  the  equation 
 \begin{equation}
\label{eq: ka=u}
\va A^{x_1} + \ldots +  \va A^{x_k}  = \vu, \qquad 1 \le x_1,  \ldots, x_k\le \tau,
\end{equation}
has a solution.

Using the orthogonality of additive characters, we can express the number $N_k(\vu)$ 
of solutions to~\eqref{eq: ka=u} as 
\begin{align*}
N_k(\vu) & = \sum_{x_1,  \ldots, x_k=1}^{\tau}  \frac{1}{q^n} 
\sum_{\vb \in  \F_q^n} \psi\(\(\va A^{x_1} + \ldots +  \va A^{x_k}  - \vu\) \vb\)\\
&  = \frac{1}{q^n}  \sum_{\vb \in  \F_q^n} S_{n,q}(\va, \vb;A)^k \psi\(-\vu \vb\).
\end{align*} 
We recall that   $\va A^x$ and $ \vu$ are  row   vectors
while $\vb$ is  column vector, thus   the multiplications above are 
well-defined. 
 
 Separating the contribution $\tau^k/q^n$ from the zero vector $\vb = \mathbf 0$ we obtain
  \begin{equation}
\label{eq: N R}
\left| N_k(\vu)  - \frac{\tau^k}{q^n}\right| \le R, 
\end{equation}
where 
$$
R =  \frac{1}{q^n}  \sum_{\vb \in \F_q^n\setminus \{\mathbf 0\}}|S_{n,q}(\va, \vb;A)|^k .
$$
Clearly we can assume that $k \ge 4$. 

By Theorem~\ref{thm:Exp Bound Irred} we now have 
$$
S_{n,q}(\va, \vb;A) \ll  \tau^{3/4-1/4n} q^{n/8},
$$
(we note this bound also holds for $ \tau> q^{n/2}$ as in this case 
$\tau^{1/2-1/4n}q^{n/4} <  \tau^{3/4-1/4n} q^{n/8}$). 
Hence, we derive 
  \begin{equation}
\label{eq:R 4th Moment} 
R \le  \frac{1}{q^n}   \( \tau^{3/4-1/4n} q^{n/8}\)^{k-4} \sum_{\vb \in  \F_q^n\setminus \{\mathbf 0\}}|S_{n,q}(\va, \vb;A)|^4 .
\end{equation}
As  in the proof  of  Theorem~\ref{thm:Exp Bound}, see~\eqref{eq:S2l}, we have
$$
\frac{1}{q^n}  \sum_{\vb \in  \F_q^n\setminus \{\mathbf 0\}}|S_{n,q}(\va, \vb;A)|^4
\le \frac{1}{q^n}  \sum_{\vb \in  \F_q^n}|S_{n,q}(\va, \vb;A)|^4 = E_{n,q}(A),
$$
(however as in the proof  of 
Theorem~\ref{thm:Exp Bound Irred} we do not need to impose the  linear independence
of $\va A^{i}$, $i =0, \ldots, n-1$). 
Therefore, by Theorem~\ref{thm:Cong Bound} we obtain
$$
R \ll   \( \tau^{3/4-1/4n} q^{n/8}\)^{k-4} \tau^{3-1/n}
= \tau^{3k/4-k/4n} q^{kn/8 - n/2}. 
$$
Hence, we see from~\eqref{eq: N R} that 
  \begin{equation}
\label{eq: N asymp}
N_k(\vu)  =  \frac{\tau^k}{q^n} \( 1 + O\(\tau^{-k/4-k/4n} q^{kn/8 + n/2}\)\).
\end{equation}
Recalling~\eqref{eq:tau large}, which we write in an equivalent form $\tau \ge q^{n^2/(2n+2) + \varepsilon}$, we see that 
$$
\tau^{-k/4-k/4n} q^{kn/8+n/2} \le q^{- (n^2/(2n+2) + \varepsilon) (k/4+k/4n) + kn/8+n/2} . 
$$
Thus, by~\eqref{eq: N asymp} it is enough to ensure that 
\begin{align*}
0 & > - (n^2/(2n+2) + \varepsilon) (k/4+k/4n) + kn/8+n/2 \\
& = - \varepsilon  (k/4+k/4n)  +n/2.
\end{align*}
This is equivalent to 
$$
k> \frac{2n}{n+1} \varepsilon^{-1} , 
$$
 which concludes the proof. 
 
 \subsection{Proof of Theorem~\ref{thm:Add Basis 2}}
We proceed as in the proof of Theorem~\ref{thm:Add Basis}
 however now we can   assume that $k \ge 7$. 

 First we consider that case when $\tau < p^{5/6} $.  
Then the last bound of Theorem~\ref{thm:Exp Bound-2-Irred} simplifies as 
$$
S_{2,p}(\va, \vb;A) \ll   \tau^{34/45}   p^{1/9}  + \tau^{8/9} \ll   \tau^{34/45}   p^{1/9}. 
$$
Hence, instead of~\eqref{eq:R 4th Moment} we derive 
$$
R \le  \frac{1}{p^2}   \( \tau^{34/45}   p^{1/9}\)^{k-6} \sum_{\vb \in  \F_p^2\setminus \{\mathbf 0\}}|S_{2,p}(\va, \vb;A)|^6 .
$$ 
As  in the proof  of  Theorem~\ref{thm:Exp Bound}, see~\eqref{eq:S2l}, we have
$$
\frac{1}{p^2}  \sum_{\vb \in  \F_p^2\setminus \{\mathbf 0\}}|S_{2,p}(\va, \vb;A)|^6
\le \frac{1}{p^2}  \sum_{\vb \in  \F_p^2}|S_{2,p}(\va, \vb;A)|^6 = F_{2,p}(A),
$$
and by Theorem~\ref{thm:3-Cong Bound-Irred}
we obtain
$$
R   \ll   \( \tau^{34/45}   p^{1/9}\)^{k-6} \(\tau^{19/5} + \tau^5 p^{-1}\)
 \ll   \( \tau^{34/45}   p^{1/9}\)^{k-6} \tau^{19/5} 
$$
since $\tau^{19/5} > \tau^5 p^{-1}$ under our assumption  $\tau < p^{5/6} $.
Hence, in the notation of  the proof of Theorem~\ref{thm:Add Basis}, we see from~\eqref{eq: N R} that 
  \begin{equation}
\label{eq: N asymp-2 a}
N_k(\vu)  =  \frac{\tau^k}{p^2} \( 1 + O\(\tau^{-11k/45 -11/15} p^{k/9 + 4/3}\)\).
\end{equation}  
Recalling the condition on $\tau$,   we see that 
 by~\eqref{eq: N asymp-2 a} it is enough to ensure that 
\begin{align*}
0 & > - (5/11 + \varepsilon) (11k/45 +11/15) + k/9 + 4/3\\
& = -11k  \varepsilon/45   -11\varepsilon/15 +1. 
\end{align*}
This is equivalent to 
$$
k> \frac{45}{11} \varepsilon^{-1} - 3 , 
$$
 which concludes the argument for $\tau < p^{5/6}$. 
 
 For  $\tau \ge  p^{5/6} $ we use the  bound 
$$
S_{2,p}(\va, \vb;A) \ll   \tau^{1/4}   p^{1/2}  
$$
of  Theorem~\ref{thm:Exp Bound-2-Irred} and also the bound~\eqref{eq:KR bound}
getting
$$  
R   \ll   \( \tau^{1/4}   p^{1/2}\)^{k-4} \tau^2
 \ll  \tau^{k/4+1}   p^{k/2 - 2} .
$$
Hence instead of~\eqref{eq: N asymp-2 a}, and recalling that  $\tau \ge  p^{5/6}$, we now obtain
\begin{align*}
N_k(\vu) & =  \frac{\tau^k}{p^2} \( 1 + O\(\tau^{-3k/4+1} p^{k/2}\)\) = 
 \frac{\tau^k}{p^2} \( 1 + O\(p^{-5k/8+5/6 + k/2} \)\)\\
& =  \frac{\tau^k}{p^2} \( 1 + O\(p^{-k/8 +5/6} \)\) =  \frac{\tau^k}{p^2} \( 1 + O\(p^{-1/24} \)\)
\end{align*}
for $k \ge 7$, which concludes the proof. 

 \section{Proof of uniformity of distribution of quantised operators} 
 
\subsection{Preliminary bounds} 

We start with deriving an explicit form of~\cite[Proposition~1]{Bourg2}, which 
could be of independent interest. 

\begin{lemma}
\label{lem: Prime}
Under the condition of Theorem~\ref{thm:QuanErg}, for a prime $N = p$ such that  
 the multiplicative order  $\tau$ of $A$ modulo $p$ satisfies $\tau \ge p^{1/2+o(1)}$, uniformly over 
$\va \in \Z^2$ such that $\va$ and $\va A$  are linearly independent modulo $p$, 
we have 
$$
\sup_{\psi \in \Psi(A)} \left | \langle  T_p(\va)\psi, \psi\rangle \right|  \le p^{-1/60+o(1)} .
$$ 
\end{lemma}

\begin{proof} We note that it is easy to see that in the argument of 
Bourgain~\cite[Equations~(2.6), (2.9) and~(2.11)]{Bourg2} one can replace $2\ell$ 
with an arbitrary integer $\nu \ge 1$, not necessary even. 
 Hence together with~\cite[Equation~(2.4)]{Bourg2}, we obtain that 
\begin{equation}
\label{eq:T and Q}
\sup_{\psi \in \Psi(A)} \left | \langle  T_p(\va)\psi, \psi\rangle \right| ^{2 \nu}
 \le \tau^{-2\nu}  p Q_{\nu,2,p}(A), 
\end{equation}
where $Q_{\nu,n,p}(A)$ is defined by~\eqref{eq:def Q}. 
We now consider two cases. 

If $\tau \ge p^{5/6}$ we choose $\nu = 2$ and use the bound~\eqref{eq:KR bound}, getting from~\eqref{eq:T and Q} that 
\begin{equation}
\label{eq:Bound T1}
\sup_{\psi \in \Psi(A)} \left | \langle  T_p(\va)\psi, \psi\rangle \right| 
 \le \(\tau^{-4} p E_{2,p}(A)\)^{1/4}  \ll  \(\tau^{-2} p\)^{1/4}  
 \le p^{-1/6}. 
\end{equation}

For $p^{5/6} > \tau \ge p^{1/2+o(1)}$ we  choose $\nu = 3$ and   use  the bound  of Theorem~\ref{thm:3-Cong Bound-Irred} 
(since the bound of Theorem~\ref{thm:3-Cong Bound-Split} is stronger), 
and so we derive
\begin{equation}
\label{eq:Bound T2}
\sup_{\psi \in \Psi(A)} \left | \langle  T_p(\va)\psi, \psi\rangle \right| 
 \le \(\tau^{-6} p F_{2,p}(A)\)^{1/6}  \ll  \(\tau^{-11/5} p\)^{1/6}  
 \le p^{-1/60+o(1)}. 
\end{equation}
Combining the bounds~\eqref{eq:Bound T1} and~\eqref{eq:Bound T2} we conclude the proof. 
\end{proof} 

We remark that in applications, such as the proof of Theorem~\ref{thm:QuanErg}
below, the quantity $o(1)$ in the condition $\tau \ge p^{1/2+o(1)}$
 of Lemma~\ref{lem: Prime},  is negative.  In fact, using the argument of  Erd\H os and  Murty~\cite[Theorem~3]{ErdMur}, one can avoid  the quantity $o(1)$ and use instead a version of  of Lemma~\ref{lem: Prime}, which under the condition   $\tau \ge p^{1/2}$
  gives the bound 
$$
\sup_{\psi \in \Psi(A)} \left | \langle  T_p(\va)\psi, \psi\rangle \right|  \ll p^{-1/60} .
$$

\subsection{Proof of Theorem~\ref{thm:QuanErg}}
The proof follows exactly the same way as the proof of~\cite[Theorem~3]{Bourg2}
however it is using the explicit bound of Lemma~\ref{lem: Prime} instead of the bound
of~\cite[Proposition~1]{Bourg2}.

Namely, the starting point is the expansion
$$ \langle  \Op_N(f)\psi, \psi\rangle  
- \int_{\T_2} f(\vv) d \vv = \sum_{\va \in \Z^2\setminus\{\mathbf{0}\}} \widehat f(\va)  \langle  T_N(\va)\psi, \psi\rangle, $$
where $\mathbf{0} = (0,0)$. 
Thus, taking into account the rapid decay of the coefficients $\widehat f(\va)$  for 
$f\in \cC^\infty \(\T_2\)$, we see that 
it is enough to obtain a bound of the form
\begin{equation}
\label{eq:Desired Bound}
\sup_{\psi \in \Psi(A)} \left | \langle  T_N(\va)\psi, \psi\rangle \right| 
 \le  \|\va\|^{O(1)} N^{-1/60+o(1)}, 
\end{equation}
where $\|\va\|$ is the Euclidean norm of $\va$.

Next, we use the inequality 
$$
 \left | \langle  T_p(\va)\psi, \psi\rangle \right|^{2\nu}
 \le \tau_N^{-2\nu} N  R_{\nu,N}\(\va, A\),
 $$ 
where $\tau_N$ is the multiplicative order of $A$ modulo $N$ and 
$R_{\nu,N}\(\va, A\)$  is the number of solutions  to the following matrix congruence
\begin{align*}
\va &\(A^{x_1} + \ldots   +A^{x_\nu} - A^{x_{\nu+1}}  - \ldots   A^{x_{2\nu}}\) \equiv \mathbf{0}^t 
\pmod N, \\
& \qquad \qquad \qquad 1  \le x_1\ldots, x_{2\nu}\le \tau_N, 
\end{align*}
see~\cite[Equations~(2.4), (2,6), (2.11) and~(3.5)]{Bourg2}.

The reduction from  estimating $R_{\nu,N}\(\va, A\)$ to estimating  $Q_{\nu,2,p}(A)$ 
as defined by~\eqref{eq:def Q}  is based on the elementary fact 
that for almost all $N$, the largest square divisor of $N$ is small and on the following 
two much more involved facts established in~\cite{KR}: 
\begin{itemize}
\item[(i)] for almost all integers  $N\ge 1$ the following two quantities
$$
\lcm[\tau_p:~p \mid N] \mand \prod_{p \mid N} \tau_p,
$$
taken over all prime divisors $p \mid N$, are of a the same order of magnitude, see~\cite[Proposition 11]{KR}; 
\item[(ii)]for almost all primes $p$, the multiplicative order $\tau_p$ of $A$ satisifies 
the inequality $\tau_p > p^{1/2+o(1)}$, see~\cite[Lemma~15]{KR}.
\end{itemize}
Thus, using~(i) and the essentially square-freeness of $N$, 
allows us to reduce bounding $ R_{\nu,N}\(\va, A\)$ to 
bounding of $Q_{\nu,n,p}(A)$ for $p\mid N$. Then  
taking $\nu = 3$, using~(ii) and arguing as in the proof  of Lemma~\ref{lem: Prime}, 
we obtain~\eqref{eq:Desired Bound} and the desired result follows.

\section{Comments}  

We note that Theorems~\ref{thm:3-Cong Bound-Split} 
and~\ref{thm:3-Cong Bound-Irred} imply upper bounds on the 6-th moments of 
Kloosterman and Gauss sums over small subgroups. More precisely, we have 
$$
\sum_{a,b\in \F_p} \left|\sfK_p(\cG; a, b) \right|^6   \ll \tau^{11/3}
$$ 
and 
$$
\sum_{a\in \F_{p^2}} \left|\sfG_{p^2}(\cG; a) \right|^6  \ll  \tau^{19/5} + \tau^5 p^{-1}. 
$$ 

Clearly Corollary~\ref{cor:Kloost}  gives bounds for exponential sums with 
Laurent binomials $aX^s + bX^{-s}$. 
It seems plausible than one can obtain versions of Corollary~\ref{cor:Kloost} 
for exponential sums with more general binomials $aX^s + bX^t$, we refer 
to~\cite{ShpVol,ShpWan} for the currently known results in this direction.

We also note that the idea behind   the proofs of Theorems~\ref{thm:3-Cong Bound-Split} 
and~\ref{thm:3-Cong Bound-Irred}, can in principle be used for matrices of higher 
dimension. However investigating irreducible factors of the corresponding bivariate 
polynomials, similar to those in Lemma~\ref{lem:AbsIrred}, however with more terms,  becomes rather difficult.  

Fimally, we note that Kelmer~\cite{Kel}  has studied  quantum ergodicity of linear maps on a  $2d$-dimensional  torus, associated with   $2d$-dimensional symplectic 
matrices, in particular, see~\cite[Propositions~3.5 and 3.6]{Kel}.  
We believe our ideas and results can be used in this setting as well.

\section*{Acknowledgement}

The authors are very grateful to  P{\"a}r Kurlberg and Ze{\'e}v Rudnick for many 
very useful comments.  

During the preparation of this work, the first two authors (A.O. and I.E.S.) were partially supported by the
Australian Research Council Grant DP200100355. The third author (J.F.V.) was partially supported by grants from the Ministry of Business, Innovation and Employment   and the Marsden Fund Council administered by the Royal Society of New Zealand.

\end{document}